\documentclass[12pt, leqno]{amsart}

\usepackage[OT2,T1]{fontenc}
\DeclareSymbolFont{cyrletters}{OT2}{wncyr}{m}{n}
\DeclareMathSymbol{\Sha}{\mathalpha}{cyrletters}{"58}

\usepackage{indentfirst}
\usepackage{amstext}
\usepackage{amsopn}
\usepackage{amsfonts}
\usepackage{amsmath}
\usepackage{latexsym}
\usepackage{amscd}
\usepackage{amssymb}
\usepackage{amsmath}
\usepackage[all,cmtip]{xy}
\usepackage{leftidx}
\usepackage{graphicx}
\usepackage{tikz}
\usepackage{hyperref}

=\oddsidemargin \font\teneufm=eufm10 \font\seveneufm=eufm7
\font\fiveeufm=eufm5
\newfam\eufmfam
\textfont\eufmfam=\teneufm \scriptfont\eufmfam=\seveneufm
\scriptscriptfont\eufmfam=\fiveeufm

\def\cF{\mathcal F}

\def\cO{\mathcal O}

\def\cE{\mathcal E}

\def\GG{\mathbb{G}}

\def\1{\mbox{\bf 1}}

\def\corad{\mathrm{corad}}
\def\rad{\mathrm{rad}}

 \DeclareMathOperator{\Hom}{Hom}
\DeclareMathOperator{\Aut}{Aut}
\DeclareMathOperator{\Autext}{Autext}

\DeclareMathOperator{\Out}{Out} 
\DeclareMathOperator{\Isom}{Isom}
\DeclareMathOperator{\Isomext}{Isomext}

\DeclareMathOperator{\Isomout}{Isomout}

\DeclareMathOperator{\GL}{\rm GL}

\newcommand{\incl}[1][r]
{\ar@<-0.2pc>@{^(-}[#1] \ar@<+0.2pc>@{-}[#1]}

\newcommand{\imm}[1][r]
   {\ar@{}[#1] |*[o][F]{\hbox{%
         }}
     \ar@{^{(}->}[#1]}

\newtheorem{stheorem}{Theorem}[section]

\newtheorem{scorollary}[stheorem]{Corollary}
\newtheorem{slemma}[stheorem]{Lemma}
\newtheorem{sproposition}[stheorem]{Proposition}
\newtheorem{sremark}[stheorem]{Remark}

\theoremstyle{definition}

\numberwithin{equation}{section}

\def\ZZ{\mathbb{Z}}

\def\cO{\mathcal{O}}

\def\Lie{\mathop{\rm Lie}\nolimits}

\def\2int{\mathop{2\int}\nolimits}

\def\Spec{\mathop{\rm Spec}\nolimits}

\def\Lie{\mathop{\rm Lie}\nolimits}

\def\Hom{\mathop{\rm Hom}\nolimits}

\def\Gal{\mathop{\rm Gal}\nolimits}

\def\scd{\mathop{\rm scd}\nolimits}

\def\Aut{\text{\rm{Aut}}}
\def\Out{\text{\rm{Out}}}

\def\Isom{\mathop{\rm Isom}\nolimits}
\def\Isomext{\mathop{\rm Isomext}\nolimits}

\def\resp.{\mathop{\rm resp.}\nolimits}
\def\limproj{\mathop{\oalign{lim\cr
\hidewidth$\longleftarrow$\hidewidth\cr}}}

\def\lgr{\longrightarrow}

\font\math=cmmi10
\def\varpi{\hbox{\math\char'44}}

\def\simlgr{\buildrel\sim\over\lgr}

\def\Coker{\mathop{\rm Coker}\nolimits}

\def\pa{\S\kern.15em }

\def\un{\uppercase\expandafter{\romannumeral 1}}
\def\deux{\uppercase\expandafter{\romannumeral 2}}
\def\trois{\uppercase\expandafter{\romannumeral 3}}
\def\quatre{\uppercase\expandafter{\romannumeral 4}}
\def\cinq{\uppercase\expandafter{\romannumeral 5}}
\def\six{\uppercase\expandafter{\romannumeral 6}}

\def\sm{\smallskip}

\def\hfl#1#2#3{\smash{\mathop{\hbox to#3{\rightarrowfill}}\limits
^{\scriptstyle#1}_{\scriptstyle#2}}}
\def\gfl#1#2#3{\smash{\mathop{\hbox to#3{\leftarrowfill}}\limits
^{\scriptstyle#1}_{\scriptstyle#2}}}

\title[Flasque resolutions]{Flasque resolutions  of linear \\
reductive  group schemes }

\date{\today}

\author{Philippe Gille}

\address[]{P. Gille, Institut Camille Jordan - Universit\'e Claude Bernard Lyon 1
43 boulevard du 11 novembre 1918,
69622 Villeurbanne cedex - France }
\email{gille@math.univ-lyon1.fr}

\begin{document}

 \begin{abstract}  This paper studies flasque resolutions for linear reductive group schemes over arbitrary bases, extending Colliot-Thélène's theory.
A flasque resolution exists for any linear reductive group scheme over 
a connected base scheme (resp.\ a quasi-compact, quasi-separated base scheme).
The construction relies on Borovoi's fundamental group and descent theory involving automorphisms
of Chevalley groups.
Applications include R-equivalence for reductive group schemes over rings and over a field of cohomological dimension $\leq 2$.
\end{abstract}

\maketitle

\bigskip


\vskip-4mm


\bigskip

\bigskip

\section{Introduction} 
Let $k$ be a field and let $G$ be a reductive $k$-group. Colliot-Th\'el\`ene constructed a flasque resolution \cite{CT2008}
$$
1 \to F \to \widetilde G \to G \to 1 
$$
generalizing the case of tori.
Here $F$ is a flasque $k$-torus and 
$\widetilde G$ is quasi-trivial.
It means that the derived group $D\widetilde G$  
is semisimple simply connected and that the 
coradical torus $Q=\widetilde G /D\widetilde G$
is quasi-trivial. There is a way to compare
two flasque resolutions so that 
the above resolution is essentially unique 
({\it ibid}, Proposition 3.2).
Remark 6.4.1 of that paper
 asks  whether one can arrange
  $F$ and $Q$ to be split
by the minimal splitting field of 
the Galois module $\pi_1(G)$ which is 
Borovoi's algebraic fundamental group of $G$.
 A "generative pretrained transformer" suggested a sketch of proof
 \cite[Hypoth\`ese 5.5]{CT2026} and
  our purpose here is different.
  We present a  complete proof of the above question
based on an alternate approach.
It involves more descent techniques and is 
mostly based on structural results of pinned Chevalley groups.
The advantage of this method is that it extends mechanically to the case of an arbitrary connected scheme $S$,so that it also refines results of   Gonz\'ales-Avil\'es \cite{Go}.
More precisely, a similar statement holds
for a reductive $S$-group scheme $G$ which is linear (Theorem \ref{main_connected}). 
The  linearity assumption is actually necessary
and we recall that 
it is equivalent to requiring 
the isotriviality of the 
 quasi-split  form of $G$ \cite[Theorem 4.1]{Gi4}.

We also provide the construction of 
 flasque resolutions for  a linear reductive group scheme $G$ over a quasi-compact and quasi-separated
base scheme $S$ (Theorem \ref{thm_qcqs}). Finally we apply
flasque resolutions to $R$-equivalence over rings.

\medskip

 \noindent{\bf Acknowledgments.}
 We thank Jean-Louis Colliot-Th\'el\`ene
 for informing us of his running work
 and for his comments on the manuscript.

\bigskip

\smallskip

 \noindent{\bf Conventions and notation.}
  We shall use the treatise SGA3 quite intensively, in particular
  those on reductive group schemes, their automorphisms, Chevalley groups, etc... For example, $\rad(G)$ denotes the radical 
  torus of a reductive $S$-group scheme $G$, that is, 
  the maximal central $S$-subtorus of $G$ (Exp.\ XXII).
We changed $\Autext$ to $\Out$ and similarly
  $\Isomext$ to $\Isomout$. 
  
  All sheaf torsors in this paper are representable, since 
  the relevant group schemes are ind-quasi-affine; this is an application 
  of Gabber's theory, see \cite[\S 2.1]{Gi5} for details.

We recall that an $S$-group scheme $G$ is
linear if there exists an  $S$-group 
monomorphism $G \to  \GL(\cE)$ where $\cE$
 is a locally
free $\cO_S$-module of finite rank.

\section{Automorphisms of reductive group schemes and fundamental groups}

\subsection{Pinning} Let $G_0$ be  a Chevalley split reductive group scheme
 defined over $\ZZ$.

Let  $(B_0,T_0)$ be a Killing pair of  $G_0$.
It gives rise to a root datum ${\Psi_0=(M_0,  M_0^\vee, R_0, R_0^\vee)}$ where $M_0= \widehat T_0(\ZZ)$ is the character group
of $T_0$ and $R_0 \subset M_0$ is the set of roots. 
We denote by $\Delta_0$ the Dynkin diagram associated to $R_0$. 

Let  $\cE_0$ be a pinning of $G_0$ relative to $(B_0,T_0)$, 
the extra data being a set $(X_\alpha)_{\alpha \in \Delta_0}$
of eigenvectors for the adjoint action of  $T_0$ on $\Lie(G_0)$. 
We denote by $E_0$ the automorphism group of the 
pinned root data $\Psi_0$, it is also 
the subgroup of $\Aut(G_0,B_0,T_0)(\ZZ)$ 
which fixes  $\cE_0$.
According to \cite[\S XXIV.1]{SGA3}, the functor of automorphisms of $G_0$ is representable by the  $\ZZ$-scheme in groups
$$
\Aut(G_0)= G_0^{ad} \rtimes (E_0)_\ZZ,
$$
where $(E_0)_\ZZ$ denotes the constant  $\ZZ$-group scheme associated to the abstract group $E_0$ and $G_0^{ad} =G_0/C(G_0)$ is the adjoint group scheme of $G_0$.

The derived $\ZZ$-group scheme $DG_0$ is semisimple and admits
a simply connected cover $G_0^{sc} \to G_0$.
The Killing pair $(B_0,T_0)$ defines a Killing pair  $(B_0^{sc},T_0^{sc})$
for $G_0^{sc}$. Also we have  natural maps 
$\Aut(G_0) \to \Aut(DG_0) \to \Aut(G_0^{sc})$.

\subsection{Fundamental groups}
For a reductive $S$-group scheme $H$ over a non-empty base scheme $S$,   Borovoi and Gonz\'alez-Avil\'es 
defined the fundamental $S$-group $\pi_1(H)$ which
is a finitely generated commutative twisted constant $S$-group scheme
\cite{BG}. This construction
extends the case of a reductive algebraic group
and  is functorial in $H$ ({\it ibid}, Remark 4.2).
In particular the $S$-group scheme $\Aut(H)$ acts on $\pi_1(H)$. We recall that 
we have an exact sequence \cite[\S XXIV.1]{SGA3}
$$
1 \to H^{ad} \to H \to \Out(H) \to 1
$$
where $H^{ad}=H/C(H)$ is the adjoint $S$-group 
scheme  of $H$ and $\Out(H)$ a twisted constant
$S$-group scheme.   
As in Topology, we have the following.

\begin{slemma} \label{lem_outer} The adjoint group $H^{ad}$  acts trivially on  $\pi_1(H)$.
\end{slemma}

\begin{proof} The idea is as in Topology:
the continuous action of a connected Lie group 
on a discrete object must be trivial.
Since $\pi_1(H)$ 
is a finitely generated commutative twisted constant $S$-group scheme, the 
$S$-functor $\Aut( \pi_1(H))$
is representable by a twisted constant $S$-group scheme.

We deal with the action homomorphism
$a: H \to \Aut( \pi_1(H))$ of $S$-group schemes.
Let $s \in S$ be a point and let 
$\kappa(s)$ be its residue field.
Since the geometric fiber $H_{\overline{\kappa(s)}}$ is connected and $\Aut( \pi_1(H))_{\overline{ \kappa(s)}}$ is a constant algebraic group,
the homomorphism $a_{\overline{\kappa(s)}}$
is trivial and so is $a_{\kappa(s)}$.

Let $H'$ be the kernel of $a$.
Since $\Aut( \pi_1(H))$ is separated, $H'$ is a closed $S$-subscheme of $H$.
We use now that $H'_{\kappa(s)}=H_{\kappa(s)}$
for each $s \in S$, so that the map
$H' \to H$ is a surjective closed immersion.
Thus $H'=H$ and $a$ is trivial.
\end{proof}

 Lemma \ref{lem_outer} defines an action of $\Out(H)$ on $\pi_1(H)$. 
In the case of $G_0$, $\pi_1(G_0)$ is actually a constant f.g.\ commutative
$S$-group and we recall to the reader its value.
 We denote by $G_0^{sc}$ the universal cover of 
the derived $\ZZ$-group scheme $G'_0=DG_0$. The Killing pair 
$(B_0,T_0)$ defines a
Killing pair $(B'_0,T'_0)$  for $G'_0$ \cite[Exp.\ XXII, Proposition 6.2.8]{SGA3};
similary the  Killing pair 
$(B'_0,T'_0)$  of  $G'_0$ defines a Killing pair $(B_0^{sc},T_0^{sc})$ of
$G_0^{sc}$  ({\it loc. cit.}, Remarque 5.8.4).
We consider the lattice  $M_0^{sc}= \widehat T_0^{sc}(\ZZ)$ and 
denote its dual by $M_0^{sc, \vee}$. 
 
We consider  the $\ZZ$-module $$\pi_1(G_0)= \Coker\bigl(  M_0^{sc, \vee} \to M_0^\vee \bigr)
$$
and $\pi_1(G_0)_{S}$ is the fundamental group of $G_0$ in the sense of Borovoi and Gonz\'alez-Avil\'es.
Note that  $M_0^{sc, \vee} \to M_0^\vee$ is injective  and
that it comes equipped with an action of $E_0$.
More precisely we have a homomorphism $E_0 \to \Aut(\pi_1(G_0))$
and we denote by  $E_{0, \sharp}$ its image in $\Aut(\pi_1(G_0))$.

We have the inclusion $\ker\bigl( E_0 \to E_{0, \sharp} \bigr) 
\subseteq \ker\bigl( E_0 \to \Aut(\rad(G_0)) \bigr)$.
Since the second kernel is finite \cite[Exp.\ XXI, Proposition 2.15]{SGA3}
so is the first one.

\subsection{Finite subgroups of $E_0$}
Let $\Gamma_\sharp$ be a finite subgroup of $E_{0, \sharp}$
and we denote by $\Gamma \subset E_0$ its preimage in $E_0$
which is finite as well. 
We denote by $$
\Aut_\Gamma(G_0)= G_0^{ad} \rtimes \Gamma_\ZZ.
$$
the closed $\ZZ$-subgroup scheme of $\Aut(G_0)$.
It is smooth affine over $\ZZ$.

Let $G$ be an $S$-form of $G_{0,S}$. It gives rise
to the $\Aut(G_0)$-torsor $P= \Isom(G_{0,S}, G)$
and to  the $E_0$-torsor $E=\Isomout(G_{0,S}, G)= P \wedge^{\Aut(G_0)} E_0$.
and  to  the $E_{0,\sharp}$-torsor $E_\sharp=\Isomout_\sharp(G_{0,S}, G)= E \wedge^{E_0} E_{0, \sharp}$.
Also the fundamental group 
$\pi_1(G)$ identifies then with the twist $^{E_\sharp}\!\pi_1(G_0)$.
Next the morphism $P \to E$ induces  an isomorphism of $S$-schemes $$
P/\Aut_\Gamma(G_0) \simlgr E/\Gamma
\simlgr E_\sharp/\Gamma_\sharp.
$$

\begin{slemma}\label{lem_reduction}
 The following are equivalent:
 
 \sm
 
 (i)  The $\Aut(G_0)$-torsor $P=\mathrm{Isom}(G_0,G)$ admits
a reduction to $\Aut_\Gamma(G_0)$;

\sm

(ii) The $E_0$-torsor $E=\mathrm{Isomout}(G_0,G)$ admits
a reduction to $\Gamma$;

\sm

(iii) The $E_{0,\sharp}$-torsor $E_\sharp=\mathrm{Isomout}_\sharp(G_0,G)$ admits
a reduction to $\Gamma_\sharp$.

 \sm
 
\noindent Furthermore, in that case, $G$ is linear.

\end{slemma}

\begin{proof}
 In view of \cite[III, Proposition 3.2.1]{Gd} 
 (i) is equivalent to  $(P/ \Aut_{\Gamma})(S) \not = \emptyset$
 and similarly (ii) is equivalent to 
 $(E/\Gamma)(S) \not = \emptyset$. Since $P/\Aut_\Gamma(G_0) \simlgr E/\Gamma$, 
 the statement follows.
 
The equivalence  $(ii) \Longleftrightarrow (iii)$
uses the same argument. 
 Finally, if (ii) holds, $G$ is linear according to \cite[Theorem 4.1]{Gi5}.  
 
\end{proof}

\section{Flasque resolutions}

\subsection{Definitions}
For tori, the definition goes back to \cite[\S 1]{CTS2}.
An $S$-torus $T$ is flasque (resp.\ coflasque,  permutation, invertible)
if  for any connected component $Z$ of $S$  there exists 
a connected \'etale cover $Z' \to  Z$, 
 principal Galois with group $H$, which splits $T$, and such that 
 $\widehat   T(Z’)$ is a flasque
(resp.\ coflasque,  permutation, invertible) $H$-module.
The definition is stable under arbitrary base changes ({\it ibid}, Proposition 1.4).

A reductive $S$-group scheme $G$ is called quasi-trivial (resp.\ coflasque) if  $DG$ is a semisimple simply-connected $S$-group scheme and the coradical torus $\corad(G)$
is a quasi-trivial (resp.\ coflasque) $S$-torus.

A flasque resolution of a  reductive $S$-group scheme $G$
is an exact sequence $1 \to F \to \widetilde G \to G \to 1$
where $F$ is a flasque $S$-torus and $\widetilde G$ a quasi-trivial $S$-group scheme.

\subsection{The equivariant Chevalley case}

Let ${\Gamma_\sharp \subset E_{0, \sharp} \subset \Aut(\pi_1(G_0))}$ 
 be a finite subgroup and let  $\Gamma \subset E_0$ be its preimage. 
The idea is to consider a resolution of $\Gamma_\sharp$-modules
(and then of $\Gamma$-modules)
$$
0 \to   N^\vee \to  M^\vee  \to  \pi_1(G_0) \to 0
$$ 
where $M$ and $N$ are lattices\footnote{It is a left action hence
the action of $\Gamma_\sharp$ on $N$ is on the right.}. We pull-back that resolution by  
the map $ M_0^\vee \to \pi_1(G_0)$ and obtain the following commutative exact diagram
of $\Gamma$-modules 
\begin{equation}\label{diag1}
\xymatrix{
&&0 \ar[d] &  0 \ar[d] \\
&& M_0^{sc,\vee}\ar@{=}[r] \ar[d]  &  M_0^{sc,\vee} \ar[d] \\
0 \ar[r] & N^\vee  \ar@{=}[d]\ar[r] & M_1^\vee  \ar[r] \ar[d] &  M_0^{\vee} \ar[d]   \ar[r] & 0 \\
    0 \ar[r] & N^\vee  \ar[r] & M^\vee  \ar[r] \ar[d] &\Pi_1(G_0)  \ar[d]  \ar[r] & 0 \\
    &&0& 0 
  }
\end{equation}
where $M_1$ is a $\Gamma$-lattice. We put $T_1=D_\ZZ(M_1)$ and
we have an exact sequence 
$$
1 \to D_\ZZ(N) \to T_1 \xrightarrow{p} T_0 \to 1
.
$$  
We define the root datum $\Psi_1=(M_1, M_1^\vee, R_1, R_1^\vee)$ as follows:

\medskip

(i) $R_1$ is the image of $R_0$ by the map $p^\vee : M_0 \to M_1$;

\sm

(ii) $R_1^\vee$  is the image of $R_0^\vee$ by the map $ M_0^{sc,\vee} \to M_1$.

\medskip

\noindent It is immediate to check that $\Psi_1$ is a root datum and also
that $p^\vee: M_0 \to M_1$ induces a morphism  $\Psi_0 \to \Psi_1$ 
of root data. Note that the order on $\Psi_0$ defines an order on $\Psi_1$ and in particular a base $\Delta_1$ of $\Psi_1$.
We denote by $E_1$ the automorphism group of the 
pinned root data $\Psi_1$.

\begin{slemma}\label{lem_action}
 (1)  The action of $\Gamma$ on $M_1$ induces
an action of $\Gamma$ on the 
pinned root data $\Psi_1$.

\sm

\noindent (2)  For the above action, the morphism  $\Psi_0 \to \Psi_1$ 
of pinned root data is $\Gamma$-equivariant.

\end{slemma}

\begin{proof} (1) Clearly $\Gamma$ acts on $R_1$ (resp.\ $R_1^\vee)$
and preserves $\Delta_1$.

\sm

\noindent (2)  It follows from the fact that $p^\vee: M_0 \to M_1$ 
is $\Gamma$-equivariant.
\end{proof}

Let $G_1$ be a split reductive $\ZZ$-group scheme with maximal split $\ZZ$-torus
$T_1$ and root datum $\Psi_1$ \cite[Exp.\ XXV, Corollaire 1.2]{SGA3};
we choose a pinning   $(Y_\alpha)_{\alpha \in \Delta_1}$
of eigenvectors for the adjoint action of  $T_1$ on $\Lie(G_1)$.
Altogether, we have a pinning $\cE_1$ for $G_1$.

In view of Chevalley's dictionnary ({\it ibid}, Th\'eor\`eme 1.1), the morphism
$\Psi_0 \to \Psi_1$  of pinned root data
defines a unique  homomorphism $p: G_1 \to G_0$ 
of reductive $\ZZ$-group schemes with a compatibility between $\cE_0$ and
$\cE_1$ (see \cite[D\'efinition 3.5.5]{D65} for details).

\begin{sremark}{\rm We note that there is a choice for $G_1$ and its pinning. Given another choice of   pinned  $\ZZ$-group scheme $G'_1$
providing $p': G'_1 \to G_0$, then there is a unique isomorphism
$u: G_1 \simlgr G'_1$ of pinned split reductive $\ZZ$-group schemes
such that $p=p' \circ u$. In other words, the choice of $G_1$ and its pinning is unique up to unique isomorphism.  
}
\end{sremark}

Lemma \ref{lem_action}.(1) provides an embedding $\Gamma \subseteq E_1$
so that we can consider  
 $\Aut_\Gamma(G_1)= G_1^{ad} \rtimes \Gamma= G_0^{ad} \rtimes \Gamma
 =\Aut_\Gamma(G_0)$ inside $\Aut(G_1)$.
Again Chevalley's equivalence of categories yields that the 
homomorphism $p:G_1 \to G_0$ is $\Aut_\Gamma(G_1)=\Aut_\Gamma(G_0)$-equivariant.
 Summarizing we obtained the following construction.

\begin{sproposition}\label{prop_michel}
 We have the
commutative exact diagram of reductive $\ZZ$-group schemes 

\[\xymatrix{
&&1 \ar[d] &   \\
&& G_0^{sc} \ar[d]  &  \\
1 \ar[r] & D_\ZZ(N)\ar[r]  & G_1  \ar[r]^p \ar[d] &  G_0   \ar[r] & 1 \\
     &  & D_\ZZ(M)   \ar[d] & \\
    &&0&  
  }
\]
which is $\Aut_\Gamma(G_1)=\Aut_\Gamma(G_0)=G_0^{ad} \rtimes \Gamma$-equivariant. Furthermore
$\Aut_\Gamma(G_0)$ acts on  $D_\ZZ(M)$ (resp.\ $D_\ZZ(N)$) 
through  the action of $\Gamma_\sharp$ 
 on $M$ (resp.\ $N$).

\end{sproposition}

\begin{proof} 
Only the last point, the equivariance, remains to be checked.
The $S$-group $G_1^{ad} \simlgr G_0^{ad}$ acts
trivially on the central $S$-subtorus $D_\ZZ(M)$
of $G_1$ and on the coradical $S$-torus  $\corad(G_1)=D_\ZZ(N)$.
On the other hand, $\Gamma$
acts on  $D_\ZZ(M)$ (resp.\ $D_\ZZ(N)$) 
through  the action of $\Gamma_\sharp$ 
 on $M$ (resp.\ $N$). Thus
$\Aut_\Gamma(G_0)$ acts on  $D_\ZZ(M)$ (resp.\ $D_\ZZ(N)$) 
through  the action of $\Gamma_\sharp$ 
 on $M$ (resp.\ $N$).
\end{proof}
 
\newpage

\subsection{The connected case}

\begin{stheorem}\label{main_connected} Let $S$ be a connected scheme and let $G$
be a reductive $S$-group scheme. Let $G_0$ be  the 
Chevalley form of $G$ and  
consider the $E_{0, \sharp}$-torsor
$E_\sharp =\Isomout_\sharp(G_0,G)$.

\smallskip

\noindent (1) The following conditions are equivalent:

\sm

(i) $G$ is linear;

\sm

(ii) The $E_{0,\sharp}$-torsor
 $E_\sharp$ is split by a finite Galois cover;

\sm

(iii) The $S$-module $\pi_1(G)$ is split by a
finite Galois cover;

\sm

(iv) $G$ admits a flasque resolution.

\smallskip

\noindent (2) Furthermore, assume that 
the $S$-module $\pi_1(G)$ is split by a
finite Galois cover $S'/S$.
Then $G$ admits a flasque resolution 
$$
1 \to F \to \widetilde G \to G \to 1
$$
such that $F$ and $\corad(\widetilde G)$ are split by $S'/S$.
\end{stheorem}

\begin{proof}
(1) Let $s$ be a quasi-geometric point of $S$ and consider the 
associated Grothendieck fundamental group $\pi_1(S,s)$ \cite[Tag0BQ8]{Stacks}.

\sm

\noindent $(iv) \Longrightarrow (i)$.
We assume that $G$ admits a 
flasque resolution $1 \to F \to \widetilde G \to
G \to 1$. Then the coradical torus
$Q=\widetilde G/ D\widetilde G$ is split by a connected Galois cover
$S'$  of $S$. Let $\widetilde 
C$ be the radical $S$-torus of 
$ \widetilde G$. We claim that $C_{S'}$ is 
split as well.
The induced map $\widetilde C \to Q$
is an isogeny whose kernel $\mu$ is finite of multiplicative type. Since $S$ is connected,
$\mu$ is of finite exponent $n$ so that 
there exists an isogeny $Q \to \widetilde C$.  
According to \cite[Lemma A.2.(ii)]{GN},
$\widetilde C_{S'}$ is split.
Since the radical $C$ of $G$ is a quotient 
of $\widetilde C$, the same argument shows
that $C_{S'}$ is a split torus.
Finally
the application of the linearity criterion
\cite[Theorem 4.1]{Gi4} shows that 
$G$ is linear.

\sm

\noindent $(i) \Longrightarrow (ii)$.
It is a special case of Lemma \ref{lem_reduction}.

\sm

\noindent $(ii) \Longrightarrow (iii)$.
We have seen that $\pi_1(G)$ is the twist
of $\pi_1(G_0)$ by $E_{0,\sharp}$.
If $E_{0,\sharp}$ is split by a Galois cover,
so is the $S$-module $\pi_1(G)$.

\sm

\noindent $(iii) \Longrightarrow (iv)$.
We use that the map $E_0 \to \Aut(\pi_1(G_0))$ is injective so that the map
$H^1(X, E_0) \to H^1(X,\Aut(\pi_1(G_0)) )$
has trivial kernel for each connected 
$S$-scheme $X$.  It implies that 
$E_{\sharp,X}$ is trivial if and only 
if $\pi_1(G)_X$ is isomorphic
to  $\pi_1(G_0)_X$.

Our assumption is that  $\pi_1(G)$ becomes
isomorphic to $\pi_1(G_0)$ after a Galois cover 
$S'/S$ so that  $E_{\sharp,S'}$ is trivial.  
Let $s'$ be a quasi-geometric point of
$S'$ above $s$. It gives rise
to a continuous homomorphism
$\pi_1(S,s) \to \Gal(S'/S)$.

The  $S$-torsors under $E_{0,\sharp}$
trivialized by $S'/S$ are classified
by $\Hom_{ct}( \Gal(S'/S), E_{0,\sharp})/E_{0,\sharp}$
so that $E_\sharp$ corresponds to some continuous homomorphism
$\phi: \pi_1(S,s) \to \Gal(S'/S) \to E_{0,\sharp}$.
We denote by 
$\Gamma_{\sharp}$ the (finite) image of $\phi$
and by $\Gamma$ its preimage in $E_0$.

Since the $E_{0, \sharp}$-torsor $E_\sharp$ admits a reduction to 
the subgroup 
$\Gamma_{\sharp}$, Lemma \ref{lem_reduction} shows that 
the $\Aut(G_0)$-torsor $P=\Isom(G_0,G)$ admits a reduction to 
some
$\Aut_\Gamma(G_0)$-torsor $Q$. We consider now the $\Aut_\Gamma(G_0)$-sequence
of reductive $\ZZ$-group schemes of Proposition \ref{prop_michel}
$$
1 \to D_\ZZ(N) \to G_1 \to G_0 \to 1.
$$
Twisting it by the $S$-torsor $Q$ provides a resolution
$$
1 \to F  \to \widetilde  G \to G \to 1. 
$$
By construction, $\widehat F(S')$ is a $\Gamma_\sharp$-module which is flasque,
$D\widetilde G$ is simply connected and  $\widehat{\corad(\widetilde G) }(S')$ is a 
permutation $\Gamma_\sharp$-module. Thus $F$ is a flasque $S$-torus split by $S'/S$
and similarly $\corad(\widetilde G)$ is a permutation $S$-torus. 

\sm

\noindent (2) It follows from (1) by inspection of the proof of  $(iii) \Longrightarrow (iv)$.

\end{proof}

\begin{scorollary}\label{cor_connected} Let $S$ be a  scheme whose connected components are clopen. Let $G$ be a reductive $S$-group scheme. 
Then $G$  is linear
if and only if $G$ admits a flasque resolution.
\end{scorollary} 

\begin{proof} It follows 
from Theorem \ref{main_connected}.(1), 
$(i) \Longleftrightarrow (iv)$, applied
to each connected component of $S$.
\end{proof}

\subsection{The case of a qcqs scheme}

\begin{stheorem}\label{thm_qcqs}
 Assume that $S$ is a quasi-compact and quasi-separated scheme.
 Let $G$ be an $S$-reductive group scheme which is linear.
 Then $G$ admits a flasque resolution.
\end{stheorem}

\begin{proof}
 We perform the  noetherian reduction trick.
According to \cite[Tag 01ZA]{Stacks}, there
exists a directed set $I$ and an inverse system of schemes 
$(S_i, f_{ii'} )$ over $I$ such that

\sm

(a) the transition morphisms $f_{ii'}$ are affine;

\sm

(b) each $S_i$ is of finite type over $\ZZ$ (in particular noetherian and qcqs), and

\sm

(c) $S = \limproj_{i \in I} S_i$. 
 
 \sm
 
 We use now  Noetherian reduction techniques 
 so that we can assume without loss of generality
  that $G=  G_{i_0} \times_{S_{i_0}} S$ where
   $G_{i_0}$ is an affine $S_{i_0}$-group scheme  of finite presentation  \cite[\S VI$_B$.10.1]{SGA3}.
 Lemmas \ref{lem_linear} and
 \ref{lem_reductive}  enable us to assume
 furthermore that each 
  $G_{i}$ is reductive and linear.
Since $S_{i_0}$ is of finite type over $\ZZ$, its connected components are clopen (combine \cite[Tags 04MF and 04ME]{Stacks}).
Corollary \ref{cor_connected} applies
and shows that $G_{i_0}$ admits a flasque
resolution. Since the flasque property is stable under base change,
we conclude that 
 $G= G_{i_0} \times_{S_{i_0}} S$  admits a flasque resolution. 
\end{proof}

\section{Application of flasque resolutions to $R$-equivalence}

Flasque resolutions of tori were introduced
by Colliot-Th\'el\`ene and Sansuc for 
studying $R$-equivalence for algebraic tori 
\cite{CTS1} and this was generalized
to reductive algebraic groups \cite[Th\'eor\`eme 8.1]{CT2008}.

With A.~Stavrova,  we  generalized  $R$-equivalence
to group schemes over  a ring 
in connection with flasque resolutions in  the torus case \cite[\S 3.2]{GS}. 
The purpose of this section is to deal with 
reductive group schemes over a ring.

\subsection{Definition}
Let $B$ be a ring (unital, commutative).
We denote by $\Sigma_B$ the multiplicative subset
of polynomials $P \in B[T]$ satisfying $P(0), P(1) \in B^\times$.
Note that evaluation at $0$ (and $1$) extends from $B[t]$
to the localization $B[t]_\Sigma$.

Let $\cF$ be a $B$-functor in sets.
We say that two points $x_0,x_1 \in \cF(B)$ are \emph{directly $R$--equivalent} if
there exists
$x \in \cF\bigl( B[t]_\Sigma \bigr)$
such that  $x_0=x(0)$ and $x_1=x(1)$.
The $R$-equivalence on $\cF(B)$ is the equivalence relation
generated by this elementary relation.

\medskip

This construction is functorial in $B$
and if $\cF$ is a group functor,
$\cF(B)/R$ carries a natural group structure
 \cite[\S 1]{GS}.

\subsection{Flasque resolution}

Let $G$ be a reductive $B$-group scheme
assumed linear. According to Theorem \ref{thm_qcqs}, $G$ admits a flasque resolution
\begin{equation}\label{seq150}
1 \to F \to \widetilde G \to G \to 1.
\end{equation}
It gives rise to the following sequence of pointed sets
\cite[Proposition 3.4.3]{Gd}  

\begin{equation}\label{diag162}
1 \to F(B) \to \widetilde G(B) \to G(B) \xrightarrow{\varphi} H^1(B,F) \to H^1(B, \widetilde G) ;
\end{equation}
the map $\varphi$ is called the characteristic map.
The point is that $\varphi$
defines a map of $B$-functors between $G$
and the $B$-functor $\cF(C)= H^1(C,F)$.
By functoriality, we get the following commutative diagram 
\begin{equation}\label{diag163}
\xymatrix{
G(B)\ar[d] \ar[r]^{\varphi} & H^1(B,F) \ar[d] \\
G(B)/R \ar[r] \ar@{.>}[ru]& \cF(B)/R .
  }
\end{equation}

\noindent If   $H^1(B,F) \to \cF(B)/R$ is an isomorphism
(equivalently the $R$-equivalence is the trivial
relation on $H^1(B,F)$),
then $\varphi$ factorizes through $G(B)/R$.
To continue we then make a stronger assumption on 
the base ring $B$.

\begin{slemma} \label{lem_regular} Assume that $B$ is a regular integral domain. For a flasque $B$-torus as above, the map   $H^1(B,F) \to \cF(B)/R$ is an isomorphism.
\end{slemma}

\begin{proof} We show that the $R$-equivalence
is the trivial relation on $H^1(B,F)$.
We are given a class $\gamma \in 
H^1(B[t]_{\Sigma_B},F)$ and want to 
show that $\gamma(0) = \gamma(1)$.
Since \'etale cohomology for $H^1$
 commutes with limits \cite[Theorem 2.1]{Mg},
 there exists $f \in \Sigma_B$
 such that $\gamma$ comes by base change from
  a class $\gamma_\sharp \in  H^1(B[t]_f,F)$.
  We consider now the maps
  $$
  H^1(B,F) \to H^1(B[t],F)  \to  H^1(B[t]_f,F) .
  $$
  The first map is an isomorphism in view
  of \cite[Lemma 2.4]{CTS2} and the second is onto
  ({\it ibid}, Theorem 2.2.(i)).
  It follows that the map $H^1(B,F) \to H^1(B[t],F)  \to  H^1(B[t]_f,F)$ is onto so that
  $\gamma_\sharp$ and $\gamma$ are constant classes.
  Thus $\gamma(0)=\gamma(1)$.
\end{proof}

We extend now \cite[Th\'eor\`eme 8.1]{CT2008}.

\begin{sproposition}\label{prop_regular}
Assume that $B$ is a regular integral domain.
Then the flasque resolution \eqref{seq150}
induces an exact sequence of groups
$$
\widetilde G(B)/R \to 
G(B)/R \to \ker\bigl(H^1(B,F) \to H^1(B,\widetilde G) \bigr).
$$
\end{sproposition}

\begin{proof} 
The proof is the same as in the field case.
Lemma \ref{lem_regular} and the above discussion show that the characteristic
map $\varphi: G(B) \to H^1(B,F)$
factorizes through $G(B)/R$. We can complete
the diagrams \eqref{diag162} 
and \eqref{diag163}  as follows

\begin{equation}
\xymatrix{
\widetilde G(B) \ar[d] \ar[r] & G(B)\ar[d] \ar[r]^{\varphi} & H^1(B,F) \ar[d] \ar[r] 
& H^1(B,\widetilde G)\\
\widetilde G(B)/R \ar[r] & G(B)/R 
\ar[r] & \cF(B)/R .
 }
\end{equation} 

\noindent By diagram chase we obtain the following exact
sequence of pointed sets
$$
\widetilde G(B)/R \to 
G(B)/R \to \ker\bigl(H^1(B,F) \to H^1(B,\widetilde G)\bigr).
$$
Since $\varphi$ is a group homomorphism,
$\ker\bigl(H^1(B,F) \to H^1(B,\widetilde G)\bigr)$
is a commutative group.
\end{proof}

\subsection{Case of fields of separable dimension
$\leq 2$}

We deal now with a base field $k$.
The notation $\scd_l(k)$ is for separable
dimension at a prime $l$, see \cite[\S 1.3.1]{Gi2238} for the definition. It coincides with the usual
cohomological dimension if $l$ is invertible 
in $k$. We slightly refine 
 \cite[Th\'eor\`eme 8.4]{CT2008}.

\begin{stheorem} \label{thm_scd}
Let $G$ be a semisimple $k$-group
and assume that $\scd_l(k) \leq 2$ for each torsion prime
of $DG$. Let 
$$
1 \to F  \to \widetilde G \to G \to 1
$$
be a flasque resolution of $G$.
\smallskip

\noindent (1) The map $D\widetilde G(k)/R \to 
\widetilde G(k)/R$ is onto.

\sm

\noindent (2) The sequence of Proposition \ref{prop_regular} induces an exact sequence
of groups
$$
D\widetilde G(k)/R \to 
G(k)/R \to H^1(k,F) \to 1.
$$ 
\end{stheorem}

\begin{proof}
(1) We denote by $E$ the coradical torus of 
$\widetilde G$,  which is a quasi-trivial 
$k$-torus. For the first part, it is 
enough to prove that the map
$R\widetilde G(k) \to E(k)$ is onto.
We define the commutative $k$-functor
$\cF(A)= \mathrm{Im}\bigl( \widetilde G(A) \to E(A) \bigr)$ for each $k$-algebra $A$.
It is then enough to show that 
$R\cF(k)=E(k)$.  

 The key tool
is the norm principle (which does not use the assumption on the field) stating that 
$$
N_{L/k}(R\cF(L)) \subset R\cF(k)
$$
for each finite separable extension $L/k$
(and more generally for each \'etale algebra), see
\cite[Theorem 3.9]{Me}.

We decompose $E$ into a product of 
quasi-trivial tori $E_i=R_{k_i/k}(\GG_m)$ and consider the preimage $\widetilde G_i$
of $E_i$ in $\widetilde G$; we are reduced to the case $E=R_{K/k}(\GG_m)$ for a separable
extension $K/k$. We write $K \otimes_k K= K \times A$ where $A$ is an \'etale $k$-algebra.

Let $X$ be the variety of Borel subgroups of
$\widetilde G$. Let $L/K$ be a finite separable
extension such that $X(L) \not = \emptyset$.
Then we have a decomposition 
$K \otimes_k L = L \times B$ compatible
with  the decomposition
$K \otimes_k K= K \times A$.
The norm principle yields
 the following commutative diagram

\begin{equation}\label{diag41}
\xymatrix{
R\cF(k) \ar@{^{(}->}[r] & E(k) = K^\times \\
R\cF(K) \ar[u] \ar@{^{(}->}[r] &
 E(K) = (K\otimes_k K)^\times  \ar[u]_{N_{K/k}}
 & = K^\times \times A^\times
\ar[ul]_{id \times N_{A/K}}   
 \\
R\cF(L)\ar[u] \ar@{^{(}->}[r] & E(L) = 
(K\otimes_k L)^\times \ar[u]_{N_{L/K}}
& = L^\times \times  B^\times
 \ar[u]_{N_{L/K} \times N_{B/A}}  
  }
\end{equation}
Since $\widetilde G_{L}$ is quasi-split, 
it admits a Borel $L$-subgroup $Q$.
Since the the coradical torus of $Q$ is an induced torus, we have $H^1(L,Q) \simlgr H^1(L,\corad(Q))=1$ \cite[Exp.\ XXIV, Corollaire 3.14]{SGA3}.
Since 
the map $E(L) \to H^1(L,D\widetilde G)$ factorizes
through  $H^1(L,Q)=1$, we obtain that $R\cF(L)=E(L)$ and 
in particular that $L^\times \times 1 \subseteq
R\cF(L)$.
By taking all such $L$'s above, we obtain
that $N_X(K) \times 1  \subseteq R\cF(K)$.
According to  \cite[Th\'eor\`eme 5.3.1]{Gi2238},
we have $N_X(K)=K^\times$.
It follows that  $K^\times \times 1  \subseteq R\cF(K)$ and the diagram yields 
that $E(k) \subset R\cF(k)$. 
Thus $R\cF(k)=E(k)$.
\bigskip

\smallskip

\noindent (2) Taking into account
Proposition \ref{prop_regular} and
(1), it remains to show that 
the map $H^1(k,F) \to H^1(k, \widetilde G)$
is trivial. Since $H^1(k,F)$ is of bounded exponent, say $d$, the map
$H^1_{fppf}(k, {_dF}) \to H^1(k,F)$ is onto where
${_dF}$ stands for the $d$-torsion of $F$.
It is then enough to establish the triviality
of the map $H^1_{fppf}(k,{_dF}) \to H^1(k,\widetilde G)$. Since this fact  is proven in \cite[Proposition 5.4.1]{Gi2238}, we are done.
\end{proof}

The next step is to determine
the group $D\widetilde G(k)/R$.
This group is expected to vanish
under the assumptions of 
Theorem \ref{thm_scd}.
There are many known cases in terms of 
groups and fields, see \cite[Th\'eor\`eme 8.4]{CT2008} and  \cite[\S 8.2]{Gi2009}.

\appendix

\section{Permanence properties}

The setting is the following.
Let $S_0$ be a base scheme (e.g.\ $S_0=\Spec(\ZZ)$).
Consider an inverse system of $S_0$-schemes 
$(S_i, f_{ii'} )$ over $I$ (where  $I$
is  a directed set)   whose  transition morphisms $f_{ii'}$ are affine.
Let $S$ be the inverse limit of the $S_i$'s.
The following lemma is known if the $S_i$'s are affine   \cite[Corollary 3.1.11]{C}.

\begin{slemma}\label{lem_reductive} 
Let $i_0 \in I$ and 
assume that $S_{i_0}$ is quasi-compact.
Let $G_{i_0}$ be an affine $S_{i_0}$-group scheme
of finite presentation. Then the following are equivalent:

\sm

(i) There exists $i_1 \geq i_0$ such that $S_i$-group scheme $G_{i_0} \times_{S_{i_0}} S_i$ is reductive (resp.\ semisimple) for each $i \geq i_1$;

\sm

(ii) The $S$-group scheme $G_{i_0} \times_{S_{i_0}} S$ is reductive (resp.\ semisimple).
 
\end{slemma}

\begin{proof} The generalization  from the affine case is immediate
by covering $S_{i_0}$ by finitely many 
affine open subsets. 
\end{proof}

\begin{slemma}\label{lem_linear} 
Let $i_0 \in I$ and  let $G_{i_0}$ be an affine $S_{i_0}$-group scheme
of finite presentation.
Then the following are equivalent:

\sm

(i) There exists $i_1 \geq i_0$ such that $G_{i_0}
\times_{S_{i_0}} S_{i_1}$ is $S_{i_1}$-linear;

\sm

(ii) $G$ is linear.

\end{slemma}

\begin{proof}  
We put $G_i= G_{i_0} \times_{S_{i_0}} S_i$
for each $i$
so that $G= \limproj_{i \geq i_0} G_i$
in view of \cite[Proposition 8.2.5]{EGA4}. 

\sm

\noindent  $(i) \Longrightarrow (ii)$. Obvious.

\smallskip

\noindent $(ii) \Longrightarrow (i)$.
We are given a monomorphism $f: G \to \GL(\cE)$ where 
$\cE$ is a locally free $\cO_{S}$-module of 
finite rank. According to 
\cite[Lemma 0B8W]{Stacks}, there exists
an index $i_1$ such that $\cE= \cE_{i_1} \otimes_{\cO_{S_{i_1}}} \cO_{S}$ and 
$\cE_{i_1}$ is a locally free $\cO_{S_{i_1}}$-module of 
finite rank. For each $i \geq i_1$, we
set $\cE_i= \cE_{i_1} \otimes_{\cO_{S_{i_1}}} \cO_{S_i}$.
Applying \cite[Exp.\ VI$_B$, Lemme 10.2]{SGA3}
shows that there exists $i_2 \geq i_1$
and a homomorphism $f_{i_2}: G_{i_2} \to 
\GL( \cE_{i_2})$ such that
 $f=f_{i_2} \times_{S_{i_2}} S$.    

 According to \cite[Tag 07RQ]{Stacks}, there exists $i_3 \geq i_2$
 such that 
 $f_{i_3}= f_{i_2}  \times_{S_{i_2}} S_{i_3}: 
 G_{i_3} \to \GL( \cE_{i_3})$ 
 is a monomorphism, hence 
 the $S_{i_3}$-group scheme $G_{i_3}$ is linear.

\end{proof}

\bigskip

\medskip

\end{document}